\documentclass[a4paper,12pt, reqno]{amsart}
\usepackage{amsmath}
\usepackage{amssymb}
\usepackage{amsxtra}
\usepackage{amsthm}
\usepackage{mathrsfs}
\usepackage{mathtools}
\usepackage{graphicx}
\usepackage{epstopdf}
\usepackage{subfigure}
\usepackage{float}
\usepackage{xcolor,graphicx,float}
\usepackage[colorlinks,linkcolor=red,anchorcolor=blue,citecolor=blue]{hyperref} \usepackage{enumerate}
\def\EquationsBySection{\def\theequation
	{\thesection.\arabic{equation}}%
	\@addtoreset{equation}{section}}

\usepackage{indentfirst}
\newtheorem{remark}{Remark}[section]

\newtheorem{theorem}[remark]{Theorem}

\newtheorem{lemma}[remark]{Lemma}

\newcommand\old[1]{}
\makeatother
\EquationsBySection

\title[A Direct Method of Moving Planes for Logarithmic Schr\"odinger Operator] %Use the shortened version of the full title
      {A Direct Method of Moving Planes for Logarithmic Schr\"odinger Operator}
       \author{Rong Zhang, Vishvesh Kumar and Michael Ruzhansky}
\address[ Rong Zhang]{HLM, Academy of Mathematics and Systems Science,\\
Chinese Academy of Sciences, Beijing 100190, P. R. China \newline and \newline Department of Mathematics: Analysis, Logic and Discrete Mathematics, Ghent University, Ghent, Belgium}
\email{zhangrong@nnu.edu.cn}
\address[Vishvesh Kumar]{Department of Mathematics: Analysis, Logic and Discrete Mathematics, Ghent University, Ghent, Belgium}
\email{vishveshmishra@gmail.com / vishvesh.kumar@ugent.be}

\address[Michael Ruzhansky]{Department of Mathematics: Analysis, Logic and Discrete Mathematics, Ghent University, Ghent, Belgium\newline and \newline
School of Mathematical Sciences, Queen Mary University of London, United Kingdom}
\email{michael.ruzhansky@ugent.be}



\allowdisplaybreaks
\begin{document}
\maketitle
\begin{abstract} 
In this paper, we study the radial symmetry and monotonicity of nonnegative solutions to nonlinear equations involving the logarithmic Schr$\ddot{\text{o}}$dinger operator $(\mathcal{I}-\Delta)^{\log}$ corresponding
to the logarithmic symbol $\log(1 + |\xi|^2)$, which is a singular integral operator given by
$$(\mathcal{I}-\Delta)^{\log}u(x)
=c_{N}P.V.\int_{\mathbb{R}^{N}}\frac{u(x)-u(y)}{|x-y|^{N}}\kappa(|x-y|)dy,$$
where $c_{N}=\pi^{-\frac{N}{2}}$, $\kappa(r)=2^{1-\frac{N}{2}}r^{\frac{N}{2}}\mathcal{K}_{\frac{N}{2}}(r)$
and $\mathcal{K}_{\nu}$ is the modified Bessel function of the second
kind with index $\nu$. The proof hinges on a direct method of moving planes for the logarithmic Schr$\ddot{\text{o}}$dinger operator.

\end{abstract}

% Enter the first author's name and address:
%\centerline{\scshape Rong Zhang$^{a,b}$,\ Vishvesh Kumar$^{b}$,\ Michael Ruzhansky$^{b,c}$}
%\medskip
%{\footnotesize
% please put the address of the first author
% \centerline{$^{a}$ School of Mathematical Sciences,
%Nanjing Normal University}
%   \centerline{ Nanjing, 210023, China}
%   \centerline{$^b$ Department of Mathematics: Analysis, Logic and Discrete Mathematics, }
 %  \centerline{ Ghent University, B 9000 Ghent, Belgium}
%   \centerline{$^c$ School of Mathematical Sciences, Queen Mary University of London, }
 %  \centerline{ Mile End Rd, Stepney, London E1 4NS, United Kingdom}
%} % Do not forget to end the {\footnotesize by the sign }

%\bigskip

% The name of the associate editor will be entered by an editorial staff
% "Communicated by the associate editor name" is not needed for special issue.
 %\centerline{}

\section{Introduction}

The study of Schr\"odinger equations received a great deal of attention from researchers in the past decades because of its vast applications in several areas of mathematics and mathematical physics. 
In particular, Schr\"odinger equations arise in quantum field theory and in the Hartree-Fock theory (see \cite{r1,r2,r3,r4}).
Recently, there is a surge of interest to investigate integrodifferential operators of order close to zero and associated linear and nonlinear integrodifferential equations (see \cite{8,9,30,20,24,27}).  In particular, the logarithmic Laplacian and the logarithmic Schr$\ddot{\text{o}}$dinger operator are two interesting examples of such a class of operators. The logarithmic Laplacian was first introduced by Chen and Weth in \cite{8} as a limit of fractional
 Laplacian (see also \cite{chy2,chy3} for the spectral properties of the logarithmic Laplacian).   The logarithmic Schr$\ddot{\text{o}}$dinger operator $(\mathcal{I}-\Delta)^{\log}$ (see \cite{aml2}) and the logarithmic Laplacian $L_{\Delta}$ (see \cite{8, Fra1,teu,rrr}) have the similar behavior locally concerning to the singularity of kernels but  
the logarithmic Schr$\ddot{\text{o}}$dinger operator eliminates the
integrability problem of the logarithmic Laplacian at infinity. To define the logarithmic Schr$\ddot{\text{o}}$dinger operator, let us begin with the following observation
\begin{equation}\label{a2}
\lim_{s\rightarrow0^{+}}(\mathcal{I}-\Delta)^{s}u(x)=u(x),\ \text{for}\ u\in C^{2}(\mathbb{R}^{N} ),
\end{equation}
where for $s\in(0,1)$, the operator $(\mathcal{I}-\Delta)^{s}$ stands for the relativistic Schr$\ddot{\text{o}}$dinger operator, for sufficiently regular function $u:\mathbb{R}^{N}\rightarrow\mathbb{R}$,  which can be represented via hypersingular integral \eqref{a2} (see \cite{10}),
\begin{equation}\label{a3}
(\mathcal{I}-\Delta)^{s}u(x)=u(x)+c_{N,s}\lim_{\epsilon\rightarrow0^{+}}\int_{\mathbb{R}^{N}\setminus B_{\epsilon}(0)}\frac{u(x)-u(y)}{|x-y|^{N+2s}}\varpi_{s}(|x-y|)dy,
\end{equation}
where $c_{N,s}=\frac{\pi^{-\frac{N}{2}}4^{s}}{\Gamma(-s)}$ is a normalization constant and the function $\varpi_{s}$ is given by
\begin{equation}\label{a4}
\begin{aligned}
\varpi_{s}(r)&=2^{1-\frac{N+2s}{2}}r^{\frac{N+2s}{2}}\mathcal{K}_{\frac{N+2s}{2}}(r)=\int_{0}^{+\infty}t^{-1+\frac{N+2s}{2}}e^{-t-\frac{r^{2}}{4t}}dt.
\end{aligned}
\end{equation}

Furthermore, if $u\in C^{2}(\mathbb{R}^{N})$, then $(\mathcal{I}-\Delta)^{s}u(x)$ is well defined by \eqref{a3} for every $x\in\mathbb{R}^{N}$. Here
the function $\mathcal{K}_{\nu}$ is the modified Bessel function of the second kind with index $\nu>0$ and it is
given by
$$\mathcal{K}_{\nu} (r)=\frac{(\frac{\pi}{2})^{\frac{1}{2}}r^{\nu}e^{-r}}{\Gamma(\frac{2\nu+1}{2})}
\int_{0}^{\infty}\big(1+\frac{t}{2}\big)^{\nu-\frac{1}{2}}e^{-rt}t^{\nu-\frac{1}{2}}dt,$$
for more properties of $\mathcal{K}_{\nu}$, see e.g. \cite{aml2, g1,g3,g2,guo} and the references
therein.

It is well known that  $\mathcal{K}_{\nu}$ is a real and positive function satisfying 
\begin{equation}\label{kmu}
\mathcal{K}'_{\nu}(r)=-\frac{\nu}{r}\mathcal{K}_{\nu}(r)-\mathcal{K}_{\nu-1}(r)<0,
\end{equation}
for all $r>0$, $\mathcal{K}_{\nu}=\mathcal{K}_{-\nu}$ for $\nu>0$. Furthermore, for $\nu>0$ (see \cite{g2,guo})
\begin{equation}\label{kk}
\mathcal{K}_{\nu}(r)\sim
\begin{cases}
\frac{\Gamma(\nu)}{2}(\frac{r}{2})^{\nu},& r\rightarrow0,\\[2mm]
\frac{\sqrt{\pi}}{\sqrt{2}}r^{-\frac{1}{2}}e^{-r},& r\rightarrow\infty.
\end{cases}
\end{equation}

%It follows from the asymptotic approximation of modified Bessel functions that there exists a large $R_{\infty} >0$ such that 
%\begin{equation} \label{asymp}
%\frac{c_\infty}{r^{\frac{1}{2}} e^r} \leq K_{\nu}(r) \leq  \frac{C_\infty}{r^{\frac{1}{2}} e^r}
%\end{equation} 
%for $r \geq R_\infty,$ where $c_\infty$ and $C_\infty$ are positive constants.

It follows from \eqref{a2} that one may expect a Taylor expansion with respect to parameter $s$ of the operator $(\mathcal{I}-\Delta)^{s}$ near zero for $u\in C^{2}(\mathbb{R}^{N})$ and $x\in\mathbb{R}^{N}$ as
\begin{equation}\label{a5}
\begin{aligned}
(\mathcal{I}-\Delta)^{s}u(x)=u(x)+s(\mathcal{I}-\Delta)^{\log}u(x)+o(s),\quad \text{as}\ s\rightarrow0^{+}.
\end{aligned}
\end{equation}
The logarithmic Schr$\ddot{\text{o}}$dinger operator $(\mathcal{I}-\Delta)^{\log}$ appears as the first order term in the
above expansion.

In this paper, we study the integrodifferential operator $(\mathcal{I}-\Delta)^{\log}$ corresponding
to the logarithmic symbol $\log(1 + |\xi|^2)$, which is a singular integral operator given by
\begin{equation}\label{intlog}
    (\mathcal{I}-\Delta)^{\log}u(x)
=c_{N}P.V.\int_{\mathbb{R}^{N}}\frac{u(x)-u(y)}{|x-y|^{N}}\kappa(|x-y|)dy,
\end{equation}
where $c_{N}=\pi^{-\frac{N}{2}}\Gamma(\frac{N}{2})$, $P.V.$ stands for the Cauchy
principal value of the integral, $\kappa(r)=2^{1-\frac{N}{2}}r^{\frac{N}{2}}\mathcal{K}_{\frac{N}{2}}(r)$
and $\mathcal{K}_{\nu}$ is the modified Bessel function of second
kind with index $\nu$. One can also easily deduce from \eqref{kmu} that $\kappa'(r)<0$ for $r>0$.
Using the expression \eqref{intlog}, one can define $(\mathcal{I}-\Delta)^{\log}$ for a quite  large class of functions $u$. To illustrate this, define the space $\mathcal{L}_{0}(\mathbb{R}^{N})$ as the space of locally integrable functions $u:\mathbb{R}^{N}\rightarrow\mathbb{R}$ such that 
$$\|u\|_{\mathcal{L}_{0}(\mathbb{R}^{N})}:= \int_{\mathbb{R}^{N}}\frac{|u(x)|e^{-|x|}}{(1+|x|)^{\frac{N+1}{2}}}dx<+\infty .$$
Then, it was proved by \cite[Proposition 2.1]{aml2} that for $u\in \mathcal{L}_{0}(\mathbb{R}^{N}) \cap L^\infty(\mathbb{R}^N)$ which is also Dini continuous at some $x\in \mathbb{R}^{N}$, the quantity $[(\mathcal{I}-\Delta)^{\log}u](x)$ is well defined by the formula \eqref{intlog}. Let us recall the definition of Dini continuity.
Let $U$ be a measurable subset of $\mathbb{R}^N$ and let $u: U\rightarrow \mathbb{R}$ be a measurable function. The modulus of continuity $\Psi_{u,x,U}: (0,+\infty)\rightarrow[0,+\infty)$ of $u$ at a point $x\in U$ is defined by
$$ \Psi_{u,x,U}(r):=\sup_{y\in U,|x-y|\leq r}|u(x)-u(y)|.$$
We call the function $u$  Dini continuous at $x$ if 
$$\int_{0}^{1}\frac{\Psi_{u,x,U}(r)}{r}dr<\infty.$$

Using the generalized direct method of moving planes, in this note we obtain the radial symmetry and monotonicity of nonnegative solutions for the nonlinear equations involving the logarithmic Schr$\ddot{\text{o}}$dinger operator (see Theorem \ref{th1}), namely,  we consider the nonlinear Schr$\ddot{\text{o}}$dinger equation
\begin{equation}\label{c1}
(\mathcal{I}-\Delta)^{\log}u(x)+mu(x)=u^{p}(x),\quad x\in \mathbb{R}^{N},
\end{equation}
with $m>0$ and $u(x)\geq0$ for all $x\in \mathbb{R}^{N}$.

The following result present symmetry and monotonicity properties of  the Schr$\ddot{\text{o}}$dinger equation \eqref{c1}.
\begin{theorem}\label{th1}
Let $u\in \mathcal{L}_{0}(\mathbb{R}^{N})$ be a nonnegative Dini continuous solution of \eqref{c1} with $m>0$ and $1<p<\infty$. If 
\begin{equation}\label{c2}
\lim_{|x|\rightarrow\infty}u(x)=a<\left(\frac{m}{p}\right)^{\frac{1}{p-1}},
\end{equation}
then $u$ must be radially symmetric and monotone decreasing about some point in $\mathbb{R}^{N}$.
\end{theorem}

\begin{remark}
    The condition \eqref{c2} in Theorem \ref{th1} is necessary for applying the method of moving planes using the decay at infinity principle (Theorem \ref{le2}).
\end{remark}\label{re}

The paper is organized as follows: In Section 2, we prove some results for the logarithmic Schr$\ddot{\text{o}}$dinger operator. 
By the direct method of moving planes, we obtain the symmetry and monotonicity of nonnegative solutions for the nonlinear equations involving logarithmic Schr$\ddot{\text{o}}$dinger operator in Section 3.

\section{Key ingredients for the method of moving planes }
This section is devoted to developing basic and key results needed to apply the method of
moving planes for establishing the proof of our main result in the next section. We first present some basic notation and nomenclatures which will be beneficial for rest
of the paper.

Choose an arbitrary direction, say, the $x_{1}$-direction. For arbitrary $\lambda\in\mathbb{R}$, let
$$T_{\lambda}=\{x\in\mathbb{R}^{N}\mid x_{1}=\lambda\}$$
be the moving plane, and let
$$\Sigma_{\lambda}=\{x\in\mathbb{R}^{N}\mid x_{1}<\lambda\}$$
be the region to the left of the plane $T_{\lambda}$, and
$$x^{\lambda}=(2\lambda-x_{1},x_{2},\cdot\cdot\cdot,x_{N})$$
be the reflection of $x$ about the plane $T_{\lambda}$. 

By denoting $u(x^{\lambda}):=u_{\lambda}(x)$, we define
$$\omega_{\lambda}(x):=u_{\lambda}(x)-u(x),\quad x\in\Sigma_{\lambda},$$
to compare the values of $u(x)$ and $u_{\lambda}(x)$.

The following results on the strong maximum principle for the operator $(\mathcal{I}-\Delta)^{\log}$ can be deduced from   \cite[Theorem 1.1]{sjtw} (see also \cite{Fra2} and \cite[Theorem 6.1]{aml2}).

\begin{lemma} \label{le} (Strong maximum principle)
Let $\Omega\subset \mathbb{R}^{N}$ be a domain, and let $u\in \mathcal{L}_{0}(\mathbb{R}^{N})$ be a continuous function on 
$\bar{\Omega}$ satisfying
\begin{equation}\label{ccc}
\begin{cases}
\ (\mathcal{I}-\Delta)^{\log}u(x)\geq0 ,& \text{in}\ \Omega,\\
\ \qquad u(x)\geq0 ,& \text{in}\ \mathbb{R}^{N}\backslash \Omega,
\end{cases}
\end{equation}
then $u>0$ in $\Omega$ or $u=0$ a.e. in $\mathbb{R}^{N}$.

\end{lemma}

Now, we will prove the following maximum principles for the logarithmic Schr$\ddot{\text{o}}$dinger operator.

\begin{theorem} \label{le1} (Maximum principle for antisymmetric functions)
Let $\Omega$ be a bounded domain in $\Sigma_{\lambda}$. Assume that $\omega_{\lambda}\in \mathcal{L}_{0}(\mathbb{R}^{N}) \cap L^\infty(\mathbb{R}^N)$ is Dini continuous on $\Omega$  and is lower semi-continuous on $\bar{\Omega}$. If
\begin{equation}\label{b1}
\begin{cases}
\ (\mathcal{I}-\Delta)^{\log}\omega_{\lambda}(x)\geq0 ,& in\ \Omega,\\
\ \qquad \omega_{\lambda}(x)\geq0 ,& in\ \Sigma_{\lambda}\backslash \Omega,\\
\ \quad \omega_{\lambda}(x^{\lambda})=-\omega(x), & in\  \Sigma_{\lambda},
\end{cases}
\end{equation}
then 
\begin{equation}\label{bb2}
\omega_{\lambda}\geq0\quad in\ \Omega.
\end{equation}
Furthermore, if $\omega_{\lambda}(x)=0$ at some point in $\Omega$, then we have 
\begin{equation}\label{b2}
\omega_{\lambda}=0,\ a.e.\ \text{in}\ \mathbb{R}^{N}.
\end{equation}
These conclusions hold for unbounded region $\Omega$ if we further assume that
$$\liminf_{|x|\rightarrow\infty}\omega_{\lambda}(x)\geq0.$$
\end{theorem}

\begin{proof}
If $\omega_{\lambda}$ is not nonnegative on $\Omega$, then the lower semi-continuity of $\omega_{\lambda}$ on $\bar{\Omega}$ implies that there exists a $x^{o}\in\bar{\Omega}$ such that
$$\omega_{\lambda}(x^{o}):=\min_{\bar{\Omega}}\omega_{\lambda}(x)<0.$$

One can further deduce from \eqref{b1} that $x^{o}$ is in the interior of $\Omega$. It follows that
\begin{equation} \label{ron4}
\begin{aligned}
(\mathcal{I}-\Delta)^{\log}\omega_{\lambda}(x^{o})
=&c_{N}P.V.\int_{\mathbb{R}^{N}}\frac{\omega_{\lambda}(x^{o})-\omega_{\lambda}(y)}{|x^{o}-y|^{N}}\kappa(|x^{o}-y|)dy\\
=&c_{N}P.V.\Bigg(\int_{\Sigma_{\lambda}}\frac{\omega_{\lambda}(x^{o})-\omega_{\lambda}(y)}{|x^{o}-y|^{N}}\kappa(|x^{o}-y|)dy\\&\quad\quad\quad\quad+\int_{ \Sigma_{\lambda}}\frac{\omega_{\lambda}(x^{o})-\omega_{\lambda}(y^{\lambda})}{|x^{o}-y^{\lambda}|^{N}}\kappa(|x^{o}-y^{\lambda}|)dy\Bigg).
\end{aligned}
\end{equation}
Since $|x^{o}-y|\leq|x^{o}-y^{\lambda}|$ we have $\frac{1}{|x^{o}-y|}\geq\frac{1}{|x^{o}-y^{\lambda}|}$ and $\kappa(|x^{o}-y|)\geq\kappa(|x^{o}-y^{\lambda}|)$ as $\kappa$ is a decreasing function, and, therefore, 
$$\frac{\omega_{\lambda}(x^{o})-\omega_{\lambda}(y)}{|x^{o}-y|^{N}}\kappa(|x^{o}-y|)\leq\frac{\omega_{\lambda}(x^{o})-\omega_{\lambda}(y)}{|x^{o}-y^{\lambda}|^{N}}\kappa(|x^{o}-y^{\lambda}|),$$
since $\omega_{\lambda}(x^{o})-\omega_{\lambda}(y)\leq0$.

Thus, we obtain from \eqref{ron4} that
\begin{equation}\label{b3}
\begin{aligned}
(\mathcal{I}-\Delta)^{\log}\omega_{\lambda}(x^{o})
\leq&c_{N}P. V.\int_{\Sigma_{\lambda}}\bigg(\frac{\omega_{\lambda}(x^{o})-\omega_{\lambda}(y)}{|x^{o}-y^{\lambda}|^{N}}
+\frac{\omega_{\lambda}(x^{o})+\omega_{\lambda}(y)}{|x^{o}-y^{\lambda}|^{N}} \bigg)\kappa(|x^{o}-y^{\lambda}|)dy\\
=&c_{N}P. V.\int_{\Sigma_{\lambda}}\frac{2\omega_{\lambda}(x^{o})}{|x^{o}-y^{\lambda}|^{N}}\kappa(|x^{o}-y^{\lambda}|)dy
<0, \end{aligned}
\end{equation}
which contradicts \eqref{b1}. Therefore, our assumption is wrong and consequently, we have   $\omega_{\lambda}(x)\geq0$ in $\Omega$.

Now we have proved that $\omega_{\lambda}(x)\geq0$ in $\Omega$. If there is some point $\tilde{x}\in\Omega$ such that $\omega_{\lambda}(\tilde{x})=0$, then from  Lemma \ref{le}, we derive immediately $\omega_{\lambda}=0$ a.e. in $\mathbb{R}^{N}$.

For unbounded domain $\Omega$, the condition
$$\liminf_{|x|\rightarrow\infty}\omega_{\lambda}(x)\geq0,$$
ensures that the negative minimum of $\omega_{\lambda}$ must be attained at some point
$x^{o}$, then we can derive the same contradiction as above.

This completes the proof of Theorem \ref{le1}.
\end{proof}
The following decay at infinity will also be necessary for proving subsequent
results.
\begin{theorem} \label{le2} (Decay at infinity)
Let $\Omega$ be an unbounded domain in $\Sigma_{\lambda}$. Suppose  that a Dini continuous $\omega_{\lambda}\in \mathcal{L}_{0}(\mathbb{R}^{N}) \cap L^\infty(\mathbb{R}^N)$ is a solution to
\begin{equation}\label{b4}
\begin{cases}
\ (\mathcal{I}-\Delta)^{\log}\omega_{\lambda}(x)+c(x)\omega_{\lambda}(x)\geq0 ,& in\ \Omega,\\
\ \qquad \omega_{\lambda}(x)\geq0 ,& in\ \Sigma_{\lambda}\backslash \Omega,\\
\ \quad \omega_{\lambda}(x^{\lambda})=-\omega(x), & in\  \Sigma_{\lambda},
\end{cases}
\end{equation}
with the measurable function $c(x)$ such that
\begin{equation}\label{b5}
\liminf_{|x|\rightarrow\infty}|x|^{\frac{1+N}{2}}c(x)\geq0.
\end{equation}
Then there exists a constant $R_{o}>0$, such that if
\begin{equation}\label{b6}
\omega_{\lambda}(x^{o})=\min_{\Omega}\omega_{\lambda}(x)<0,
\end{equation}
then
\begin{equation}\label{b7}
|x^{o}|\leq R_{o}.
\end{equation}

\end{theorem}

\begin{proof}
We prove the assertion by contradiction. Suppose that \eqref{b7} is false, then by \eqref{b4} and \eqref{b6}, we have
$$\omega_{\lambda}(x^{o})=\min_{\Sigma_{\lambda}}\omega_{\lambda}(x)<0.$$

After a direct calculation, we obtain
\begin{equation}\label{b8}
\begin{aligned}
(\mathcal{I}-\Delta)^{\log}\omega_{\lambda}(x^{o})
=&c_{N}P.V.\int_{\mathbb{R}^{N}}\frac{\omega_{\lambda}(x^{o})-\omega_{\lambda}(y)}{|x^{o}-y|^{N}}\kappa(|x^{o}-y|)dy\\
=&c_{N}P.V.\int_{\Sigma_{\lambda}}\bigg(\frac{\omega_{\lambda}(x^{o})-\omega_{\lambda}(y)}{|x^{o}-y|^{N}}\kappa(|x^{o}-y|)
+\frac{\omega_{\lambda}(x^{o})-\omega_{\lambda}(y^{\lambda})}{|x^{o}-y^{\lambda}|^{N}}\kappa(|x^{o}-y^{\lambda}|)\bigg)dy\\
\leq&c_{N}P.V.\int_{\Sigma_{\lambda}}\bigg(\frac{\omega_{\lambda}(x^{o})-\omega_{\lambda}(y)}{|x^{o}-y^{\lambda}|^{N}}
+\frac{\omega_{\lambda}(x^{o})+\omega_{\lambda}(y)}{|x^{o}-y^{\lambda}|^{N}} \bigg)\kappa(|x^{o}-y^{\lambda}|)dy\\
=&c_{N} P.V. \int_{\Sigma_{\lambda}}\frac{2\omega_{\lambda}(x^{o})}{|x^{o}-y^{\lambda}|^{N}}\kappa(|x^{o}-y^{\lambda}|)dy
<0.
\end{aligned}
\end{equation}

Now, we fix $\lambda$ and when $|x^{o}|\geq\lambda$, we have $B_{|x^{o}|}(\breve{x})\subset\tilde{\Sigma}_{\lambda}:=\mathbb{R}^{N}\backslash\Sigma_{\lambda}$ with $\breve{x}=(3|x^{o}|+x_{1}^{o},(x^{o})')$. 
Then, for $y\in\tilde{\Sigma}_{\lambda}$, if $|x^{o}|\geq\frac{R_{\infty}}{4}$ with sufficiently large $R_\infty$, we can deduce that $|x^{o}-y| \leq |x^{o}-\breve{x}|+|\breve{x}-y|\leq |x^{o}|+3|x^{o}|=|4x^{o}|$  which together with the fact that $\kappa$ is a decreasing  function implies that
$$\frac{\kappa(|x^{o}-y|)}{|x^{o}-y|} \geq \frac{\kappa(|4x^{o}|)}{|4x^{o}|}.$$

Thus, from \eqref{kk} and $\kappa(r)=2^{1-\frac{N}{2}}r^{\frac{N}{2}}\mathcal{K}_{\frac{N}{2}}(r)$, if  $R_\infty$  is sufficiently large, we have
\begin{equation}\label{b9}
\begin{aligned}
\int_{\Sigma_{\lambda}}\frac{1}{|x^{o}-y^{\lambda}|^{N}}\kappa(|x^{o}-y^{\lambda}|)dy
&=\int_{\tilde{\Sigma}_{\lambda}}\frac{\kappa(|x^{o}-y|)}{|x^{o}-y|^{N}}dy
\geq\int_{B_{|x^{o}|}(\breve{x})}\frac{\kappa(|4x^{o}|)}{|4x^{o}|^{N}}dy\\
&\geq\int_{B_{|x^{o}|}(\breve{x})}\frac{2^{1-\frac{N}{2}}\mathcal{K}_{\frac{N}{2}}(|4x^{o}|)}{|4x^{o}|^{\frac{N}{2}}}dy\\
&\geq\frac{c_{\infty}\omega_{N} }{2^{\frac{3N}{2}}|x^{o}|^{\frac{1+N}{2}}e^{4|x^{o}|}}
:=\frac{C}{|x^{o}|^{\frac{1+N}{2}}e^{4|x^{o}|}},
\end{aligned}
\end{equation}
where $C=c_{\infty}\omega_{N}2^{-\frac{3N}{2}}$ is a positive constant.

It follows that
$$0\leq(\mathcal{I}-\Delta)^{\log}\omega_{\lambda}(x^{o})+c(x^{o})\omega_{\lambda}(x^{o})\leq\bigg(\frac{C }{|x^{o}|^{\frac{1+N}{2}}e^{4|x^{o}|}}+ c(x^{o}) \bigg)\omega_{\lambda}(x^{o}),$$
or equivalently,
$$\frac{C }{|x^{o}|^{\frac{1+N}{2}}e^{4|x^{o}|}}+ c(x^{o}) \leq0.$$
Now,  if $|x^{o}|$ is sufficiently large, this would contradict \eqref{b5}. Therefore, \eqref{b7} holds. 

This completes the proof of Theorem \ref{le2}.
\end{proof}

%The following narrow region principle for the operator $(\mathcal{I}-\Delta)^{\log}$ can be deduce from a more general result  in \cite[Theorem 6.1 (iii)]{aml2}.

%\begin{theorem}
    
%\end{theorem}

\section{Proof of the main theorem}

\begin{proof}[Proof of Theorem \ref{th1}]
Let $T_{\lambda},\Sigma_{\lambda},x^{\lambda}$ and $\omega_{\lambda}$ be defined as in the previous section. Then at the points where $\omega_{\lambda}(x)<0$, it is easy to verify that, for $\xi_{\lambda}(x)\in(u_{\lambda}(x), u(x)),$ we have
\begin{equation}\label{c3}
\begin{aligned}
(\mathcal{I}-\Delta)^{\log}\omega_{\lambda}(x)+m\omega_{\lambda}(x)=u_{\lambda}^{p}(x)-u^{p}(x)=p\xi_{\lambda}^{p-1}(x)\omega_{\lambda}(x)\geq p u^{p-1}(x)\omega_{\lambda}(x),
\end{aligned}
\end{equation}
because $\omega_{\lambda}(x)<0$, and $\xi_{\lambda}(x)< u(x)$.

$\mathbf{Step\ 1}$. We will show that for sufficiently negative $\lambda$,
\begin{equation}\label{c4}
\omega_{\lambda}(x)\geq0,\quad x\in\Sigma_{\lambda}.
\end{equation}

First, from the assumption \eqref{c2}, for each fixed $\lambda$,
$\lim_{|x|\rightarrow\infty}\omega_{\lambda}(x)=0.$
In fact, by \eqref{c2}, we have $\lim_{|x|\rightarrow\infty}u(x)=a$, and $\lim_{|x|\rightarrow\infty}u_{\lambda}(x)=a$ implying that $\lim_{|x|\rightarrow\infty}\omega_{\lambda}(x)=0.$

Thus, if \eqref{c4} is false, then the negative minimum of $\omega_{\lambda}$ can be obtained at some point, say, $x^o$ in $\Sigma_{\lambda}$, that is, 
$$\omega_{\lambda}(x^{o})=\min_{\Sigma_{\lambda}}\omega_{\lambda}(x)<0.$$

Set $c(x):=m-pu^{p-1}(x)$ in \eqref{c3} and then, the assumption \eqref{c2} implies that $c \in L^\infty(\mathbb{R}^N)$ and 
$$\lim_{|x|\rightarrow\infty}c(x)\geq0.$$
Consequently, from Theorem \ref{le2} it follows that there exists $R_{o}>0$ (independent of $\lambda$), such that
\begin{equation}\label{c5}
|x^{o}|\leq R_{o}.
\end{equation}
Therefore, by choosing $\lambda<-R_{o}$ and consequently, $|x^\lambda|>R_0$ for $x \in \Sigma_{\lambda},$  we obtain  by \eqref{c5} that
\begin{equation}\label{cc5}
\omega_{\lambda}(x)\geq0,\quad x\in\Sigma_{\lambda}.
\end{equation}
$\mathbf{Step\ 2}$. Step 1 provides a starting point, from which we can now move the plane $T_{\lambda}$ to the right as long as 
\eqref{c4} holds to its limiting position.
Define
$$\lambda_{o}:=\sup\{\lambda\mid \omega_{\mu}(x)\geq0,\ \forall \ x\in\Sigma_{\mu},\ \forall \mu\leq\lambda\}.$$
By \eqref{c5}, we know $\lambda_{o}<\infty$.

Next, we will show via a contradiction argument that
\begin{equation}\label{c6}
\omega_{\lambda_{o}}(x)\equiv0,\quad \forall \ x\in\Sigma_{\lambda_{o}}.
\end{equation}
Suppose, on the contrary, that
\begin{equation}\label{c7}
\omega_{\lambda_{o}}(x)\geq0,\ \text{and}\ \ \omega_{\lambda_{o}}(x)\not\equiv0, \ \text{in}\ \Sigma_{\lambda_{o}},
\end{equation}
then we must have
\begin{equation}\label{c8}
\omega_{\lambda_{o}}(x)>0,\quad \forall \ x\in\Sigma_{\lambda_{o}}.
\end{equation}

In fact, if \eqref{c8} is violated, then there exists a point $\hat{x}\in\Sigma_{\lambda_{o}}$ such that
$$\omega_{\lambda_{o}}(\hat{x})=\min_{\Sigma_{\lambda_{o}}}\omega_{\lambda_{o}}(x)=0.$$
It means that $u_{\lambda_{o}}(\hat{x})=u(\hat{x})$.
Then it follows from \eqref{c3} that
$$(\mathcal{I}-\Delta)^{\log}\omega_{\lambda_{o}}(\hat{x})=u_{\lambda_{o}}^{p}(\hat{x})-u^{p}(\hat{x})=u^{p}(\hat{x})-u^{p}(\hat{x})=0.$$
Hence, Theorem \ref{le1} implies that $\omega_{\lambda_{o}}(\hat{x})\equiv0$ in $\Sigma_{\lambda_{o}}$, which contradicts
\eqref{c7}. Thus \eqref{c8} holds.

Now, we will show that the plane $T_{\lambda}$ can be moved further right. More precisely, there exists an $\epsilon>0$ such that, for any $\lambda\in[\lambda_{o},\lambda_{o}+\epsilon)$, we have 
\begin{equation}\label{c9}
\omega_{\lambda}(x)\geq0,\quad  x\in\Sigma_{\lambda}.
\end{equation}
Once it is proved this will  contradict the definition of $\lambda_{o}$. Therefore, \eqref{c6} must be valid.

Let us now prove \eqref{c9}. In fact, by \eqref{c8}, we have $\omega_{\lambda_{o}}(x)>0,\,\,\,  x\in\Sigma_{\lambda_{o}},$
which in turn implies that there is a constant $c_{o}>0$ and $\delta>0$ such that
$$\omega_{\lambda_{o}}(x)\geq c_{o}>0,\quad x\in\overline{\Sigma_{\lambda_{o}-\delta}\cap B_{R_{o}}(0)}.$$

Since $\omega_{\lambda}$ is continuous with respect to $\lambda$ there exists an $\epsilon>0$ such that for $\lambda\in[\lambda_{o},\lambda_{o}+\epsilon)$ we have
\begin{equation}\label{c10}
\omega_{\lambda}(x)\geq0,\quad  x\in\Sigma_{\lambda_{o}-\delta}\cap B_{R_{o}}(0).
\end{equation}
Moreover, combining \eqref{c5} with \eqref{c10}, we deduce that $w_\lambda(x) \geq 0$ on $\Sigma_{\lambda_0-\delta}.$ 

To proceed with the proof, we need the following small volume maximum principle (\cite[Theorem 6.1 (iii)]{aml2} and \cite[Theorem 1.3]{sjtw}).

\begin{lemma} \label{small} Let $\Omega$ be a open set of $\mathbb{R}^N.$ Consider the following problem on $\Omega$:
\begin{equation} \label{r610}
\begin{cases}
(\mathcal{I}-\Delta)^{\log}u(x)\geq c(x)u,& in\ \Omega,\\
\qquad\qquad u\geq0, & in\ \mathbb{R}^{N}\backslash\Omega,
\end{cases}
\end{equation}
with $ c\in L^{\infty}(\mathbb{R}^{N})$.
\\
Then, there exists $\delta>0$ such that for every open  set $\Omega \subset \mathbb{R}^{N}$ with $|\Omega|\leq\delta$ and any solution $u\in \mathcal{V}_{\omega}(\Omega)$ of \eqref{r610} in $\Omega$,
where the space $\mathcal{V}_{\omega}(\Omega)$ is given in   \cite[Section 6]{aml2}, we have $u\geq0$ in $\mathbb{R}^{N}$.
\end{lemma}

Consequently,
according to Lemma \ref{small} (by taking $\Omega=(\Sigma_{\lambda}\setminus\Sigma_{\lambda_{o}-\delta}) \cap B_{R_o}(0)$), we obtain that
\eqref{c9} holds.

The arbitrariness of the $x_{1}$-direction leads to the radial symmetry of $u(x)$ about some point in $\mathbb{R}^{N}$, and the monotonicity is a consequence of the fact that \eqref{cc5} holds.

This completes the proof of the Theorem \ref{th1}.
\end{proof}

\section*{Acknowledgments}
RZ is supported by the Postdoctoral Fellowship Program of CPSF under Grant No GZC20232913.
 VK and MR are supported by the FWO Odysseus 1 grant G.0H94.18N: Analysis and Partial
Differential Equations, the Methusalem programme of the Ghent University Special Research Fund (BOF) (Grant number 01M01021) and by FWO Senior Research Grant G011522N. MR is also supported by EPSRC grant
EP/R003025/2.

\end{document}